\documentclass[12pt,leqno]{amsart}  
\usepackage{amsthm, bbm}
\usepackage{amssymb, amsmath}

\usepackage{mathtools}
\mathtoolsset{showonlyrefs,showmanualtags}

\usepackage{hyperref} %,hypertexnames=false,colorlinks,[pagebackref]
\hypersetup{
    colorlinks=true,       % false: boxed links; true: colored links
    linkcolor=blue,          % color of internal links
    citecolor=magenta,        % color of links to bibliography
    filecolor=magenta,      % color of file links
    urlcolor=cyan           % color of external links
}

\usepackage[backrefs]{amsrefs}

\theoremstyle{plain} % definition 
\newtheorem{lemma}[equation]{Lemma} 
\newtheorem{proposition}[equation]{Proposition} 
\newtheorem{theorem}[equation]{Theorem} 
\newtheorem{corollary}[equation]{Corollary} 

\newtheorem{priorResults}{Theorem}

\theoremstyle{definition}

\theoremstyle{remark}

\numberwithin{equation}{section}

\usepackage[top=1.5in, bottom=1.5in, left=1.25in, right=1.25in]	{geometry}
 
\begin{document}
 
 \title{Sparse Bounds for the Discrete Cubic Hilbert Transform}
 \date{}
 \author[Culiuc]{Amalia Culiuc}   %  can use \and  

\address{School of Mathematics, Georgia Institute of Technology, Atlanta GA 30332, USA}
\email {amalia@math.gatech.edu}
\thanks{Research supported in part by grant  NSF-DMS-1600693}
 
 \author[Kesler]{Robert Kesler}   %  can use \and  

\address{School of Mathematics, Georgia Institute of Technology, Atlanta GA 30332, USA}
\email {rkesler6@math.gatech.edu}

 \author[Lacey]{Michael T. Lacey}   %  can use \and  

\address{ School of Mathematics, Georgia Institute of Technology, Atlanta GA 30332, USA}
\email {lacey@math.gatech.edu}
\thanks{Research supported in part by grant  NSF-DMS-1600693}
 
 \maketitle 
\begin{abstract}
 Consider the discrete cubic Hilbert transform defined on finitely supported functions $f$  on $\mathbb{Z}$ by
 \begin{eqnarray*}
 H_3f(n) = \sum_{m \not = 0} \frac{f(n- m^3)}{m}.
 \end{eqnarray*}
We prove that there exists $r <2$ and universal constant $C$ such that for all finitely supported $f,g$ on $\mathbb{Z}$  there exists an $(r,r)$-sparse form ${\Lambda}_{r,r}$ for which 
 \begin{eqnarray*}
\left| \langle H_3f, g \rangle \right| \leq C {\Lambda}_{r,r} (f,g). 
 \end{eqnarray*}
This is the first result of this type concerning discrete harmonic analytic operators. It immediately implies some weighted inequalities, which are also new in this setting.

 \end{abstract}

 %%%%%%%%%%%%%%%%%%%%%%%%%%%%%% SECTION  SECTION SECTION
%%%%%%%%%%%%%%%%%%%%%%%%%%%%%% SECTION  SECTION SECTION 
\section{Introduction} %\label{s:}

The purpose of this paper is to initiate a theory of sparse domination for discrete operators in harmonic analysis.  
We do so in the simplest non-trivial case; it will be clear that there is a much richer theory to be uncovered.  

Our main result concerns the discrete cubic Hilbert transform, defined for finitely supported functions $f$ on $\mathbb{Z}$ by 
\begin{equation*}
H_3 f (x) = \sum_{n\neq 0}\frac { f (x- n ^{3})}n .  
\end{equation*}
It is known  \cites{MR1056560,IW} that this operator extends to a bounded linear operator on $ \ell ^{p} (\mathbb Z )$ to $ \ell ^{p} (\mathbb Z )$, for all $ 1< p < \infty $.  We prove a sparse bound, which in turn proves certain weighted inequalities. Both results are entirely new.  

By an \emph{interval} we mean a set $ I = \mathbb Z \cap [a,b]$, for $ a< b \in \mathbb R $.  For $ 1\leq r < \infty $, set 
\begin{equation*}
\langle f \rangle _{I,r} := \Bigl[\frac 1 {\lvert  I\rvert } \sum_{x\in I} \lvert  f (x)\rvert ^{r}  \Bigr] ^{1/r}. 
\end{equation*}   
We say a collection of intervals $ \mathcal S $ is \emph{sparse} if there are subsets $ E_S \subset S \subset \mathbb Z $ with (a) $ \lvert  E_S\rvert > \tfrac 14 \lvert  S\rvert  $, uniformly in $ S\in \mathcal S$, and  (b) the sets $ \{E_S  :  S\in \mathcal S\}$ are pairwise disjoint.  For  sparse collections $ \mathcal S$, consider sparse bi-sublinear forms 
\begin{equation*}
\Lambda _{\mathcal S, r, s} (f,g) := \sum_{S\in \mathcal S} \lvert  S\rvert \langle f \rangle _{S,r} \langle g \rangle _{S,s}. 
\end{equation*}
 %We only consider such forms for sparse collections of intervals $ \mathcal S$. 
 Frequently we will suppress the  collection $ \mathcal S$, and if $ r=s=1$, we will suppress this dependence as well.  

The main result of this paper is the following theorem.
%%%%%%%%%%%%%%%%%%%%%%%%%%%%%% THEOREM THEOREM THEOREM
\begin{theorem}\label{t:main}  There is a choice of $ 1< r <2 $, and constant $ C>0$ so that for all $ f, g $ that are finitely supported on $\mathbb{Z}$, there is a sparse collection of intervals $ \mathcal S$, so that 
\begin{equation}\label{e:main}
\lvert  \langle H_3 f, g \rangle\rvert \leq C \Lambda _{\mathcal S, r,r} (f,g).  
\end{equation}
\end{theorem}
%%%%%%%%%%%%%%%%%%%%%%%%%%%%%% THEOREM THEOREM THEOREM

The beauty of sparse operators is that they are both positive and highly localized operators.  In particular, many of their mapping properties can be precisely analyzed.  As an immediate corollary \cite{MR3531367}*{\S 6} we obtain weighted inequalities, holding in an appropriate intersection of Muckenhoupt and reverse H\"older weight classes. 

%%%%%%%%%%%%%%%%%%%%%%%%%%%%%% COROLLARY COROLLARY COROLLARY
\begin{corollary}\label{c:wtd} There exists $ 1< r <2 $ so that for all weights $ w^{-1},  w \in A_2 \cap RH_r $ we have 
\begin{equation*}
\lVert H_3  :  \ell ^2 (w) \mapsto \ell ^2 (w)\rVert \lesssim 1 .  
\end{equation*}
\end{corollary}
%%%%%%%%%%%%%%%%%%%%%%%%%%%%%%  COROLLARY COROLLARY COROLLARY

For instance, one can take $w(x) = [1+\lvert x\rvert]^{a}$, for $-\frac12 < a < \frac 12$.  

\smallskip 

The concept of a sparse bound originated in \cites{MR3085756,MR3521084,2015arXiv150105818L}, so it is new, in absolute terms, as well as this area.  On the other hand, the study of norm inequalities for discrete arithmetic operators has been under active investigation for over 30 years. However, \emph{no weighted inequalities have ever been proved in this setting. } 

\bigskip 
The subject of discrete norm inequalities of this type began with the breakthrough work of Bourgain \cites {MR937582,MR937581} on arithmetic ergodic theorems.  He proved, for instance, the following Theorem. 

%%%%%%%%%%%%%%%%%%%%%%%%%%%%%% THEOREM THEOREM THEOREM
\begin{priorResults}\label{t:bourgain}  Let $ P   $ be a polynomial on $\mathbb Z$ which takes integer values. Then the maximal function $ M_P$ below maps $ \ell ^{p} (\mathbb Z )$ to $ \ell ^{p} (\mathbb Z)$ for all $ 1< p < \infty $. 
\begin{equation*}
M _{P} f (x) = \sup _{N} \frac 1N \sum_{n=1} ^{N} \lvert  f (x- p (n))\rvert .  
\end{equation*}
\end{priorResults}
%%%%%%%%%%%%%%%%%%%%%%%%%%%%%% THEOREM THEOREM THEOREM

Subsequently, attention turned to a broader understanding of Bourgain's work, including its implications for singular integrals and Radon transforms \cites{ISMW,MR1056560}. 
The fine analysis needed to obtain results in all $ \ell ^{p}$ spaces was developed by Ionescu and Wainger \cite{IW}. 
This theme is ongoing, with recent contributions in \cites{2015arXiv151207523M,2015arXiv151207524M,2015arXiv151207518M}, while other variants of these questions can be found in \cites{2015arXiv151206918K,MR2661174}.  

Initiated by Lerner \cite{MR3085756} as a remarkably simple proof of the so-called $A_2$ Theorem, the study of sparse bounds for operators has recently been an active topic.  The norm control provided in \cite{MR3085756} was improved to a pointwise control for Calder\'on-Zygmund operators in \cites{2015arXiv150105818L,MR3521084}. 
The paper \cite{2016arXiv160305317C} proved sparse bounds for the bilinear Hilbert transform, 
in the language of sparse forms, pointing to the applicability of sparse bounds outside the classical Calder\'on-Zygmund setting. That point of view is crucial for this paper.  
 
Two papers  \cites{2016arXiv160906364L,2016arXiv160908701K} have  proved sparse bounds for \emph{random} discrete operators,  a much easier  setting than the current one. 
A core technique of these papers reappears in \S \ref{s:minor}. 
Sparse bounds continue to be explored in a variety of settings 
\cites{2016arXiv160506401B,MR3531367,2016arXiv161103808K,2016arXiv161001531L,2015arXiv151005789H}.  

\bigskip 

We recall some aspects of known techniques in sparse bounds in \S \ref{s:general}.  These arguments and results are   formalized in a new notation, which makes the remaining quantitative proof more understandable. 
In particular, we define a `sparse norm' and formalize some of its properties.    
Our main theorem above is a sparse bound for a Fourier multiplier.  
In \S \ref{s:decompose}, we describe a decomposition of this Fourier multiplier, which has a familiar form within the discrete harmonic analysis literature.  The multiplier is decomposed into `minor' and `major' arc components, which require dramatically different methods to control.    
Concerning the minor arcs, there is one novel aspect of the decomposition, a derivative condition which has a precursor in \cite{2015arXiv151206918K}. Using this additional feature, the minor arcs are controlled in \S \ref{s:minor} through a variant of an argument in \cite{2016arXiv160906364L}.   
The major arcs are the heart of the matter, and are addressed in \S \ref{s:major}.

\bigskip 

An expert in the subject of discrete harmonic analysis will recognize that there are many possible extensions of the main result of this paper. We have chosen to present the main techniques in the simplest non-trivial example.    
Many variants and extensions to our main theorem hold, but all the ones we are aware of are more complicated than this one.

%%%%%%%%%%%%%%%%%%%%%%%%%%%%%% SECTION  SECTION SECTION
%%%%%%%%%%%%%%%%%%%%%%%%%%%%%% SECTION  SECTION SECTION 
\section{Generalities} \label{s:general}

We collect some additional notation, beginning with the one term that is not standard, namely the sparse operators.  
Given an operator $ T $ acting on finitely supported functions on $ \mathbb Z $, and index $ 1\leq r, s < \infty $, we set 
\begin{equation}\label{e:SPN}
\lVert T  : \textup{Sparse} (r,s)\rVert
\end{equation}
to be the infimum over  constants $ C>0$ so that for all finitely supported functions $ f, g$ on $ \mathbb Z $,  
\begin{equation*}
\lvert  \langle Tf, g \rangle\rvert  \leq C \sup \Lambda _{r,s} (f,g), 
\end{equation*}
where the supremum is over all  sparse forms.  In particular, the `sparse norm' in \eqref{e:SPN} satisfies a triangle inequality. 
\begin{equation}\label{e:quasi}
\Bigl\lVert \sum_{j} T_j  : \textup{Sparse} (r,s) \Bigr\rVert 
\leq \sum_{j}\lVert  T_j  : \textup{Sparse} (r,s) \rVert.  
\end{equation}
We collect some quantitative estimates for different operators, hence 
the notation.  As the notation indicates, it suffices  to exhibit a single sparse bound for $ \langle Tf,g \rangle$.  

It is known that the Hardy-Littlewood maximal function 
\begin{equation*}
M _{\textup{HL}} f = \sup _{N} \frac 1 {2N+1} \sum_{j=-N} ^{N} \lvert  f (x-j)\rvert 
\end{equation*}
satisfies a sparse bound.  This is even a classical result. 

%%%%%%%%%%%%%%%%%%%%%%%%%%%%%% THEOREM THEOREM THEOREM
\begin{priorResults}\label{t:Max}
We have 
\begin{equation*}
\lVert M _{\textup{HL}} \;:\; \textup{Sparse} (1,1)\rVert \lesssim 1. 
\end{equation*}
\end{priorResults}
%%%%%%%%%%%%%%%%%%%%%%%%%%%%%% THEOREM THEOREM THEOREM

The following is a deep fact about sparse bounds that is at the core of our main theorem.  
%%%%%%%%%%%%%%%%%%%%%%%%%%%%%% LEMMA LEMMA LEMMA
\begin{priorResults}\label{t:sparse}\cites{MR3521084,2015arXiv150105818L} Let $ T_K $ be the convolution with any Calder\'on-Zygmund kernel.  For a Hilbert space $ \mathcal H$, and viewing $ T_K$ as acting on $ \mathcal H$ valued functions, we have the sparse bound 
\begin{equation*}
 \lVert T_K    : \textup{Sparse} (1,1)\rVert < \infty .  
\end{equation*}
We make the natural extension of the definition of the sparse form to vector valued functions, namely 
$ \langle f \rangle_I = \lvert  I\rvert ^{-1} \sum_{x\in I} \lVert f\rVert _{\mathcal H} $.  
\end{priorResults}
%%%%%%%%%%%%%%%%%%%%%%%%%%%%%% LEMMA LEMMA LEMMA

Recall that $ K$ is a \emph{Calder\'on-Zygmund kernel on $ \mathbb R $} if $ K  : \mathbb R \setminus \{0\} \to \mathbb C $ 
satisfies 
\begin{align} \label{e:CZK}
 \sup _{ x \in \mathbb R \setminus \{0\}} \lvert   x K (x)\rvert + \lvert   x ^2  \tfrac d {dx}K (x)\rvert< \infty ,
\end{align}
and $ T_K$ acts boundedly from $ L^2 $ to $ L^2 $.  The kernels that we will encounter are small perturbations of $ 1/x$. 
Restricting a Calder\'on-Zygmund kernel to the integers, we have a kernel which satisfies Theorem~\ref{t:sparse}.

\smallskip 

In a different direction, we will accumulate a range of sparse operator bounds at different points of our argument.  Yet there is, in a  sense, a unique maximal sparse operator, once a pair of functions $ f, g$ are specified.  
Thus we need not specify the exact sparse form which proves our main theorem.  

%%%%%%%%%%%%%%%%%%%%%%%%%%%%%% LEMMA LEMMA LEMMA
\begin{lemma}\label{l:oneFormToRuleThem} \cite{2016arXiv161001531L}*{Lemma 4.7} 
Given finitely supported functions $ f, g$ and choices of $ 1\leq r, s< \infty $, there is a sparse form $ \Lambda ^{\ast} _{r,s}$, and constant $ C>0$ so that for any other sparse form $ \Lambda _{r,s}$ we have 
\begin{equation*}
\Lambda _{r,s} (f,g) \leq C \Lambda ^{\ast} _{r,s} (f,g).  
\end{equation*}
\end{lemma}
%%%%%%%%%%%%%%%%%%%%%%%%%%%%%% LEMMA LEMMA LEMMA

A couple of elementary estimates, which we will appeal to, are in this next proposition. 
The use of these inequalities in the sparse bound setting appeared in \cite{2016arXiv160906364L}. 
%%%%%%%%%%%%%%%%%%%%%%%%%%%%%% PROPOSITION PROPOSITION PROPOSITION
\begin{proposition}\label{p:elementary}
Let $ T _{K} f (x) = \sum_{n} K (n) f (x-n)$ be convolution with kernel $ K$.  Assuming that $ K$ is finitely supported 
on interval $ [-N,N]$
we have  the  inequalities 
\begin{align} \label{e:elem2}
\lVert T_K  : \textup{Sparse} (r,s)\rVert& \lesssim  N ^{1/r+1/s-1} \lVert T _{K}  :  \ell ^r \mapsto  \ell ^{s'}\rVert , \qquad 1\leq r, s < \infty . 
\end{align}
\end{proposition}
%%%%%%%%%%%%%%%%%%%%%%%%%%%%%% PROPOSITION PROPOSITION PROPOSITION

The two instances of the above inequality we will use are $ (r,s)= (1,1), (2,2)$. In the latter case, one should observe that the power of $ N$ above is zero.

%%%%%%%%%%%%%%%%%%%%%%%%%%%%%% PROOF PROOF PROOF
\begin{proof}
Let $ \mathcal I$ be a partition of $ \mathbb Z $ into intervals of length $ 2 N$. Assume that 
if $ I, I' \in \mathcal I$ with $ \textup{dist} (I,I')\leq 1$, then either $ f \mathbf 1_{I} $ or $ f \mathbf 1_{I'}$ are identically zero.  Then, 
\begin{align*}
\lvert \langle  T_K f, g \rangle\rvert  & 
\leq \sum_{I\in \mathcal I}  \langle f \mathbf 1_{I} , T ^{\ast}_K (g \mathbf 1_{3I}) \rangle 
\\
& \leq   
\lVert T _{K}  :  \ell ^r \mapsto  \ell ^{s'}\rVert 
\sum_{I\in \mathcal I} \lVert f \mathbf 1_{3I}\rVert_r 
\lVert g \mathbf 1_{3I}\rVert_s  
\\
& \lesssim 
N ^{1/r+1/s-1} \lVert T _{K}  :  \ell ^r \mapsto  \ell ^{s'}\rVert 
\sum_{I\in \mathcal I}  \lvert  3I\rvert \cdot  \langle f  \rangle _{3I,r} 
 \langle g \rangle _{3I,s}.
\end{align*}
\end{proof}
%%%%%%%%%%%%%%%%%%%%%%%%%%%%%% PROOF PROOF PROOF

The definition of sparse collections has a useful variant.  Let $ 0< \eta \leq \frac 14$.  
We say a collection of intervals $ \mathcal S $ is \emph{$ \eta $-sparse} if there are subsets $ E_S \subset S \subset \mathbb Z $ with (a) $ \lvert  E_S\rvert > \eta  \lvert  S\rvert  $, uniformly in $ S\in \mathcal S$, and  (b) the sets $ \{E_S  :  S\in \mathcal S\}$ are pairwise disjoint.

%%%%%%%%%%%%%%%%%%%%%%%%%%%%%% LEMMA LEMMA LEMMA
\begin{lemma}\label{l:eta}
For each $ f, g$ there is a $ \frac 12$-sparse form $ \Lambda $ so that 
for all    $ \eta $-sparse forms $ \Lambda ^{\eta } $,  we have 
\begin{equation}\label{e:eta}
 \Lambda ^{\eta } (f,g) \lesssim \eta ^{-1} \Lambda (f,g), \qquad  0< \eta < 1/4. 
\end{equation}
\end{lemma}
%%%%%%%%%%%%%%%%%%%%%%%%%%%%%% LEMMA LEMMA LEMMA

%%%%%%%%%%%%%%%%%%%%%%%%%%%%%% PROOF PROOF PROOF
\begin{proof}
Let $ \mathcal S ^{\eta }$ be the sparse collection of intervals associated to $ \Lambda ^{\eta}$. 
Using shifted dyadic grids \cite{MR3065022}*{Lemma 2.5}, 
we can, without loss of generality, assume that $ \mathcal S ^{\eta }$ consists of dyadic intervals. It follows that 
we have the uniform Carleson measure estimate 
\begin{equation*}
\sum_{J\in \mathcal S  :  J\subset I} \lvert  J\rvert \lesssim \eta ^{-1} \lvert  I\rvert, \qquad I\in \mathcal S  ^{\eta }.   
\end{equation*}
Then, for an integer $ J \lesssim \eta ^{-1} $, we can decompose $ \mathcal S ^{\eta }$ into 
subcollections $ \mathcal S _{j}$, for $ 1\leq j \leq J$, so that each collection $ \mathcal S _j$ is $ \frac 12 $-sparse.  

Now, with $ f, g$ fixed,  by Lemma~\ref{l:oneFormToRuleThem}, there is a single sparse operator $ \Lambda $ so that uniformly in $ 1\leq j \leq J $, we have 
\begin{equation*}
\Lambda _{\mathcal S_j} (f,g) \lesssim \Lambda (f,g). 
\end{equation*}
This completes our proof. 
\end{proof}
%%%%%%%%%%%%%%%%%%%%%%%%%%%%%% PROOF PROOF PROOF

A variant of the sparse operator will appear, one with a `long tails' average. 
Define 
\begin{equation}\label{e:new}
\{  f \} _{S}  =  \frac 1 {\lvert  S\rvert } \sum_{x } \frac {\lvert  f (x)\rvert } {\left( 1+\frac{ \textup{dist} (x, S)}{|S|} \right)^{3 }}.
\end{equation}

%%%%%%%%%%%%%%%%%%%%%%%%%%%%%% LEMMA LEMMA LEMMA
\begin{lemma}\label{l:new} 
For all finitely supported $ f, g$, there is a sparse operator $ \Lambda $ so that for any sparse 
collection $ \mathcal S_0$, there holds 
\begin{equation}\label{e:new<}
\sum_{S\in \mathcal S_0} \lvert  S\rvert \{f\}_S \{g\}_S \lesssim \Lambda (f,g).  
\end{equation}
\end{lemma}
%%%%%%%%%%%%%%%%%%%%%%%%%%%%%% LEMMA LEMMA LEMMA

%%%%%%%%%%%%%%%%%%%%%%%%%%%%%% PROOF PROOF PROOF
\begin{proof}
For integers $ t >0$ let $ \mathcal S _{t} = \{2 ^{t} S  :  S\in \mathcal S\}$. Assuming that $ \mathcal S_0$ is $\tfrac 12 $-sparse, it follows that  $ \mathcal S_t$ is  $ 2 ^{-t-1}$-sparse, for $ t>0$. Appealing to the power decay in \eqref{e:new}
\begin{equation*}
\sum_{S\in \mathcal S_0} \lvert  S\rvert \{f\}_S \{g\}_S  \lesssim 
\sum_{t=0} ^{\infty } 2 ^{-2t}
\Lambda _{\mathcal S_t} (f,g). 
\end{equation*}
But by Lemma~\ref{l:eta}, there is a fixed $ \tfrac 12 $-sparse form $ \Lambda (f,g)$ so that 
\begin{equation*}
\Lambda _{\mathcal S_t} (f,g) \lesssim 2 ^{t} \Lambda (f,g), \qquad t >0.  
\end{equation*}
So the proof is complete. 
\end{proof}
%%%%%%%%%%%%%%%%%%%%%%%%%%%%%% PROOF PROOF PROOF

\bigskip 
Throughout, $e(x):=e^{2\pi i x}$, and $ \varepsilon >0$ is a fixed small absolute constant.  
 For a function $f\in \ell^2(\mathbb{Z})$, the  (inverse) Fourier transform of $f$ is defined as
\begin{align*}
\mathcal{F}f(\beta)&:=\sum_{n\in \mathbb{Z}}f(n)e(-\beta n),
\\
\mathcal{F}^{-1}g(n)&=\int_\mathbb{T}g(\beta)e(\beta n)d\beta.
\end{align*}
 
 We will define operators as Fourier multipliers.  Namely, given a function $ M  : \mathbb T \mapsto \mathbb C $, we define the associated linear operator by 
 \begin{equation} \label{e:FM}
\mathcal F [\Phi _{M} f] (\beta ) = M (\beta ) \mathcal Ff (\beta ).  
\end{equation}
The notation $ \mathcal F ^{-1} M = \check M$ will be convenient.  
As above, for kernel $ K$, the operator $ T_K$ will denote convolution with respect to $ K$. Thus, 
$ \Phi _{M} = T _{\check M}$.

%%%%%%%%%%%%%%%%%%%%%%%%%%%%%% SECTION  SECTION SECTION
 \section{The Main Decomposition} \label{s:decompose}
 
 We prove the main result by decomposition of the Fourier multiplier 
 \begin{equation} \label{e:M}
M (\beta ) := \sum_{m\neq 0} \frac {e (-\beta  m ^{3})} m 
\end{equation}
In this section, we detail the decomposition, which is done in the standard way, with one new point needed.

  \textbf{The kernel.} Let $\left\{\psi_j \right\}_{j \geq 0}$ be a dyadic resolution of $\frac{1}{t}$, where $\psi_j(x) = 2^{-j} \psi(2^{-j}x)$ is a smooth odd function satisfying $|\psi(x)| \leq 1_{[1/4, 1]}(|x|)$.  In particular
 \begin{equation} \label{e:psi}
\sum_{k\geq 0}\psi_k(t) = \frac 1 t , \qquad \lvert  t\rvert\geq 1.  
\end{equation}
 
\textbf{The Major Arcs.} 
The rationals in the torus are the union over $s \in \mathbb{N}$ of the collections $\mathcal{R}_s$ given by
\begin{equation}\label{e:R}
 \mathcal{R}_s := \left\{ B/Q \in \mathbb{T}: (B, Q)=1, 2^{s-1} \leq Q < 2^s \right\}. 
\end{equation} 
 Namely the denominator of the rationals is held approximately fixed.  
 For all rationals $B/Q\in \mathcal{R}_s$, define the $j$-th major box at $ B/Q$,  to be the
 \[\mathfrak{M}_j(B/Q):=\{\beta\in \mathbb{T},  |\beta-B/Q|\leq 2^{(\varepsilon -3)j}\}, \qquad s\leq \varepsilon j. \] 
 Collect the major arcs, denoting 
\begin{eqnarray}  \label{e:major}
\mathfrak{M}_j := \bigcup_{(B, Q)=1: Q \leq 2^{6j \epsilon}} \mathfrak{M}_j(B/Q). 
\end{eqnarray}
Note in particular  that for a sufficiently small $\varepsilon$,  in the union above no two distinct major arcs $ \mathfrak{M}_j(B/Q)$ intersect.  
That is,  if $B_1/Q_1\neq B_2/Q_2$, suppose that $\beta\in \mathfrak{M}_j(B_1/Q_1)\cup \mathfrak{M}_j(B_2/Q_2)$. Then 
 \[2^{-6j\varepsilon}\leq|B_1/Q_1-B_2/Q_2|\leq |B_1/Q_1-\beta|+|B_2/Q_2-\beta|\leq 2^{(\varepsilon -3)j+1},\]
which is a contradiction for $\varepsilon<2/7$.

\textbf{Multipliers.}
We use the notation below for the decomposition of the multiplier.  
\begin{align} \label{e:Mj}
M_j(\beta) & := \sum_{m \in \mathbb{Z}} e(-\beta m^3) \psi_j(m),
\\ \label{e:Hj}
H_j(y) & := \int_\mathbb{R} e(-yt^3) \psi_j(t) dt,  \quad  \qquad  \textup{(Continuous analog of $ M_j$)}
   \\
S(B/Q) &:= \frac{1}{Q} \sum_{r = 0}^{Q-1} e(-B/Q \cdot r^3), \quad  \qquad  \textup{(Gauss sum)}
\\  \label{e:Ljs}
L_{j,s}(\beta)& :=\sum_{B/Q\in \mathcal{R}_s} S(B/Q)H_j(\beta-B/Q)\chi_s(\beta-B/Q), 
\\
\noalign{\noindent where $\chi$ is  a smooth even bump function with
$\mathbbm{1}_{[-1/10,1/10]}\leq \chi\leq\mathbbm{1}_{[-1/5,1/5]}$ and  $\chi_s(t)=\chi(10^st)$, }  
\\  \label{e:Lj}
 L_j(\beta) & := \sum_{s \leq j \epsilon} L_{j,s}(\beta), \qquad j \geq 1, 
\\\label{e:Ls}
L^s(\beta)&:=\sum_{ j  \geq s/ \varepsilon }L_{j,s}(\beta), \qquad s\geq 1, \\ 
L (\beta ) &: =  \sum_{s=1}^\infty L^{s}(\beta) = \sum_{j=1}^\infty L_j(\beta),
\\ \label{e:Ej}
E_j(\beta)&:= M_j(\beta) - L_j(\beta), \qquad j \geq 1 \\ 
 \label{e:E} 
 E(\beta) &:= \sum_{j=1}^\infty E_j(\beta).
   \end{align}
Therefore, by construction, $M(\beta) = L(\beta) + E(\beta)$ for all $\beta \in \mathbb{T}$. Our motivation for introducing the above decomposition is that the discrete multiplier $M_j$ is well-approximated by its continuous analogue $L_j$ on the major arcs in $\mathfrak{M}_j$.  And off of the major arcs, the multiplier is otherwise small.  

Theorem \ref{t:main} is proved by showing that there exists $ 1< r < 2$ and $\kappa >0$ such that
\begin{align} \label{e:Ej-est}
\lVert \Phi _{E_j}  : \textup{Sparse} (r,r)\rVert \lesssim 2 ^{- \kappa  j}, \qquad j \geq 1 \\ 
\label{e:Ls-est}
\lVert \Phi _{L^s}  : \textup{Sparse} (r,r)\rVert \lesssim 2 ^{- \kappa  s}, \qquad  s \geq 1. 
\end{align}
Indeed, from the above inequalities, it follows that 
\begin{align}
\lVert \Phi_L  :  \textup{Sparse} (r,r)\rVert 
\leq \sum_{s=1} ^{\infty } 
\lVert \Phi_{L^s}  :  \textup{Sparse} (r,r)\rVert \lesssim 
 \sum_{s=1} ^{\infty } 2 ^{- \kappa s  } \lesssim 1,
\\\label{e:ES}
\lVert \Phi_E  :  \textup{Sparse} (r,r)\rVert 
\leq \sum_{j=1} ^{\infty } 
\lVert \Phi_{E_j}  :  \textup{Sparse} (r,r)\rVert \lesssim 
 \sum_{j=1} ^{\infty } 2 ^{- \kappa j  } \lesssim 1. 
\end{align}
Therefore, our main theorem follows from 
\begin{align}
\lVert \Phi_M  :  \textup{Sparse} (r,r)\rVert & \leq 
\lVert \Phi _L  :  \textup{Sparse} (r,r)\rVert + \lVert \Phi_E  :  \textup{Sparse} (r,r)\rVert \lesssim 1. 
\end{align}
We  prove the `minor arcs' estimate \eqref{e:Ej-est} in \S{4} and the `major arcs' estimate \eqref{e:Ls-est} in \S{5}.

\bigskip
The next theorem gives quantitative estimates for the Gauss sums \eqref{e:Gauss} and the multipliers $E_j$ defined in \eqref{e:Ej} that are essential to our proof of Theorem \ref{t:main}. 

 \begin{theorem}\label{t:error}
For absolute choices of $\varepsilon >0 $, 
\begin{gather}  \label{e:Gauss}
 \lvert  S(B/Q) \rvert  \lesssim 2 ^{- \varepsilon s}, \qquad B/Q\in \mathcal R_s, \qquad s\geq 1,  
 \\   \label{e:Ej<}
\lVert E_j(\beta)\rVert _{\infty }\lesssim 2^{-\varepsilon j} , \qquad j\geq 1 , 
 \\    \label{e:D}
\Bigl\lVert \frac{d^2}{d\beta^2} E_j(\beta)  \Bigr\rVert _{\infty } \lesssim 2^{7j}, \qquad j\geq 1.  
 \end{gather}
 \end{theorem}
 
The first two are well-known estimates.  
The estimate \eqref{e:Gauss} is the Gauss sum bound, see \cite{Hua}, while the estimate \eqref{e:Ej<} is gotten by combining Lemma \ref{minorarcs} and Lemma \ref{approx}. The only unfamiliar estimate is the derivative bound \eqref{e:D}, but our claim is very weak and follows from elementary considerations.

  \bigskip 
  
The details of a proof of the Theorem \ref{t:error} are represented in the literature \cites{MR1056560,2015arXiv151206918K}.  We indicate the details.  A central lemma is this approximation of $M_j$ defined in \eqref{e:Mj}, in terms of $L_{j}$ defined in \eqref{e:Lj}.

  \begin{lemma}\label{approx}
  For $1 \leq s \leq \epsilon j, B/Q \in \mathcal{R}_s$, we have the approximation 
  \begin{eqnarray*}
  M_j(\beta) = L_{j} (\beta) + O (2^{(2 \epsilon-1)j}), \qquad  \beta \in \mathfrak{M}_j(B/Q). 
  \end{eqnarray*}
  \end{lemma}
  
 \begin{proof}
 We closely follow the argument in \cite{2015arXiv151206918K}. 
 There are two estimates to prove. 
 \begin{align}  \label{e:ML1}
 \lvert  M_j(\beta) - S(B/Q)H_j(\beta-B/Q) \rvert \lesssim 2^{(2 \epsilon-1)j}, 
 \\ \label{e:ML2}
 \lvert L _{j} (\beta) - S(B/Q)H_j(\beta-B/Q)  
 \rvert \lesssim 2^{(2 \epsilon-1)j},
 \end{align}
 both estimates holding uniformly over $\beta \in \mathfrak{M}_j(B/Q)$, and $B/Q\in \mathcal R_s$.  
 
For the second estimate \eqref{e:ML2}, it follows from the definitions of $L_j$ and $L _{j,s}$ in \eqref{e:Ljs}, as well as the disjointess of the major arcs that 
\begin{align*}
\lvert  L _{j} (\beta) &-S(B/Q)H_j(\beta-B/Q)  
 \rvert 
 \\&= 
 \lvert   L _{j,s} (\beta) -  S(B/Q)H_j(\beta-B/Q)
 \rvert 
 \\
 &\leq  
 \lvert  S(B/Q)H_j(\beta-B/Q)\rvert 
 (\mathbf{1}_{\mathfrak{M}_j (B/Q)} - \chi (10^s (\beta - B/Q)))
 \\
 & \lesssim \sup _{ \lvert \beta \rvert > \frac 12 10^{s-1}}  \lvert H_j (\beta) \rvert  \lesssim 10^{-s}. 
\end{align*}
The last bound is a standard van der Corput estimate.  
 
\medskip 

We turn to \eqref{e:ML1}.  
 Write $\beta = B/Q + \eta$, where $|\eta| \leq 2^{(\epsilon-3)j}$. For all positive $m$ in the support of $\psi_j$, decompose these integers into their residue classes $\mod Q$, i.e. $m = pQ +r$,t where $0 \leq r < Q \leq 2^{j \epsilon}$, and the $p$ values are integers in $[c,d]$, with $c= d/8 \simeq 2^j /Q$ to cover the support of $\psi_j$. 
 The argument of the exponential in \eqref{e:M} is, modulo 1, given by
 \begin{eqnarray*}
 \beta (pQ +r)^3  = (B/Q + \eta )(pQ+r)^3 \equiv r^3 B/Q +(pQ)^3 \eta + O ( 2^{j (2\epsilon-1)})
 \end{eqnarray*}
 %It is immediate that % \begin{eqnarray*}
% e(-\beta m^3) = e(-r^3 \cdot B/Q - (pQ)^3 \eta) + O(2^{j (2 \epsilon-1)}),
 %\end{eqnarray*}
 Then the sum over all positive integers  $m$ in the support of $\psi_j$ can be written as \begin{align*}
 &\sum_{p \in [c,d]} \sum_{r=0}^{Q-1} \left[ e(-r^3 B/Q -(pQ)^3 \eta) +O(2^{(2\epsilon-1)j}) \right]  \psi_j(pQ+r) \\ &= 
 \sum_{r =0}^{Q-1} e(-r^3 \cdot B/Q) \times \sum_{ p \in [c,d]} e(- \eta (pQ)^3) \psi_j(pQ) + O(2^{(2\epsilon-1)j}) \\ &= S(B,Q) \times Q  \sum_{ p \in [c,d]} e( -\eta (pQ)^3) \psi_j(pQ) + O(2^{(2\epsilon-1)j}) .
 \end{align*}
 For fixed $ p \in [c,d]$ and $0 \leq t \leq Q$, we have 
 \begin{align*}
 |e(-\eta(pQ)^3 ) &\psi_j(pQ) - e(- \eta (pQ+t)^3) \psi_j(pQ+t)| \\&\lesssim  |e(-\eta (pQ)^3) - e(-\eta(pQ+t)^3)| 2^{-j} + | \psi_j(pQ) - \psi_j(pQ+t)| \lesssim 2^{(2\epsilon-2)j}. 
 \end{align*}
 Therefore, 
 \begin{eqnarray*}
 Q  \sum_{ p \in [c,d]} e( -\eta (pQ)^3) \psi_j(pQ) = \int_0^\infty e(-\eta t^3) \psi_j(t) dt + O(2^{(2\epsilon-1)j}).
 \end{eqnarray*}
 The analogous computation for negative values of $m$ yields 
  \begin{eqnarray*}
 \sum_{ m <0} e(-\beta m^3) \psi_j(m) = S(B,Q) \times \int_{-\infty}^0 e(-\eta t^3) \psi_j(t) dt + O(2^{(2\epsilon-1)j}),
 \end{eqnarray*}
 and combining the two estimates with the notation in \eqref{e:E} leads to the desired conclusion. 
 \end{proof}
 
We also need control of $M_j$ and $L_j$, defined in \eqref{e:Lj} on the  minor arcs, which are the open components of the complement of $\mathfrak{M}_j$ defined in \eqref{e:major}.  
 
 \begin{lemma}\label{minorarcs}
 %(1) (Complete Gauss Sums) For any $0 < \nu < \frac{1}{3}$, uniformly in $B/Q$, 
 %\begin{eqnarray*}
 %|S(B/Q)| \lesssim_\nu 2^{-s(1/3{-\nu})}, \qquad (B,Q)=1. 
 %\end{eqnarray*}
 %(\textcolor{blue}{This is already stated.})  
 
 There is a $\delta = \delta(\epsilon)$ so that uniformly in $j \geq 1$, 
 \begin{eqnarray*}
 |M_j(\beta)| +|L_j(\beta)| \lesssim 2^{-\delta j}, ~~~~\beta \not \in \mathfrak{M}_j. 
 \end{eqnarray*}
 \end{lemma}
 
This estimate is essentially present in \cite{2015arXiv151206918K}. The bound $| M_j(\beta)| \lesssim 2^{-\delta j}$ for $\beta \not \in \mathfrak{M}_j$  can be seen from  Bourgain \cite{MR1019960}*{Lemma 5.4}, 
and is a consequence of a fundamental estimate of Weyl \cite{Iwaniec_Kowalski}*{Theorem 8.1}.  
The corresponding bound on $L_j$ is an easy consequence of the Van der Corput estimate  $|H_j(y)| \lesssim  2^{-j} |y|^{-1/3}$.

%%%%%%%%%%%%%%%%%%%%%%%%%%%%%% SECTION  SECTION SECTION
%%%%%%%%%%%%%%%%%%%%%%%%%%%%%% SECTION  SECTION SECTION 
 \section{Minor Arcs} \label{s:minor}
%By the `minor arcs' we mean the term $E (\beta )$ given by \eqref{e:E}.  Now, $ E (\beta )$ is a sum over $ j$ of functions $ E _{j} (\beta )$.   We prove this Lemma, 
Recalling the sparse form notation \eqref{e:SPN} and the Fourier multiplier notation \eqref{e:FM}, we now proceed to the proof of the bound in \eqref{e:Ej-est}. 

%%%%%%%%%%%%%%%%%%%%%%%%%%%%%% LEMMA LEMMA LEMMA
\begin{lemma}\label{l:minor} There exists $\kappa >0$  and $ 1< r <2 $ such that
\begin{equation}\label{e:ESparse}
\lVert \Phi _{E_j}  : \textup{Sparse} (r,r)\rVert \lesssim 2 ^{- \kappa  j}, \qquad j\geq 1.  
\end{equation}
\end{lemma}
%%%%%%%%%%%%%%%%%%%%%%%%%%%%%% LEMMA LEMMA LEMMA

%This verifies the inequalities in \eqref{e:ES}.  

%%%%%%%%%%%%%%%%%%%%%%%%%%%%%% PROOF PROOF PROOF
\begin{proof}
We only need the $ L ^{\infty }$ bound on $ E_j$ given in \eqref{e:Ej<}, and the derivative condition \eqref{e:D}.  
In particular, these two conditions imply 
\begin{equation}\label{e:Echeck}
\lvert  \mathcal F ^{-1}  E_j (m)\rvert  \lesssim \min \Bigl\{ 2 ^{- \varepsilon j},  \frac {2 ^{7j}} { 1+ m ^2 }  \Bigr\}. 
\end{equation}

Write $  \mathcal F ^{-1}  E_j  = \check E_ {j,0} + \check E_{j,1} $, where $  \check  E_ {j,0} (m) = 
[\mathcal F ^{-1}  E_ {j} (m) ]\mathbf 1_{[-2 ^{10j}, 2 ^{10j}] } (m)$.  
  It follows immediately from \eqref{e:Echeck} that 
 \begin{gather}
\lVert T _{\check E_{j,1}}  : \ell ^2 \mapsto \ell ^2  \rVert  \lesssim \lVert \check  E_ {j,1}\rVert_1 \lesssim 2 ^{-3j}, 
\end{gather}
(Recall that $T_K$ denotes convolution with respect to kernel $ K$.) 
But,  it follows that $ T_K f \lesssim M _{\textup{HL}}f $ where the latter is the maximal function.  
And so by Theorem~\ref{t:Max}, we have 
\begin{equation*}
\lVert T _{\check E_{j,1}}  :  \textup{Sparse} (1,1) \rVert \lesssim 2 ^{-3j}.  
\end{equation*}
 
It remains to provide a sparse bound for $ T_{\check E_{j,0}}$ (which is the interesting case).  
We are in a position to use \eqref{e:elem2}, with $ N \simeq 2 ^{10j}$.  We have for $ 1< r < 2$
\begin{equation} \label{e:t0r}
\lVert T _{\check E_{j,0}}  :  \textup{Sparse} (r,r) \rVert \lesssim  2 ^{10j (\frac 2r-1)} 
\lVert T_ {\check E_{j,0}}  :  \ell ^r \mapsto \ell ^{r'}\rVert . 
\end{equation}
Notice that $ \frac 2r-1$ can be made arbitrarily small.  We need to estimate the operator norm above.  
But, we have the two estimates 
\begin{align*}
\lVert T_ {\check E_{j,0}}  :  \ell ^s \mapsto \ell ^{s'}\rVert \lesssim 2 ^{- \varepsilon j} ,  \qquad s=1,2. 
\end{align*}
The case of $ s=1$ follows from \eqref{e:Echeck}, and the case of $ s=2$ from Plancherel and  \eqref{e:Ej<}.  
We therefore see that we have a uniformly small estimate on the norm of $ T_{\check{E}_{j,0}}$ from $ \ell ^{r} \mapsto \ell ^{r'}$, for $ 1< r < 2$.  For $ 0< 2-r \ll \varepsilon $, we have the desired bound in \eqref{e:t0r}.  

\end{proof}
%%%%%%%%%%%%%%%%%%%%%%%%%%%%%% PROOF PROOF PROOF

%%%%%%%%%%%%%%%%%%%%%%%%%%%%%% SECTION  SECTION SECTION
 \section{Major Arcs}  \label{s:major}
 
 The following estimate is the core of the Main Theorem.  Recalling the definition of $ L ^{s}$ in \eqref{e:Ls}, 
 the notation for Fourier multipliers \eqref{e:FM} and the sparse norm notation \eqref{e:SPN}, we have this, which verifies the bound in \eqref{e:Ls-est}.
 
\begin{lemma}\label{l:major}
There exists $\kappa >0$ and $1<r<2$ such that 
  \begin{equation}\label{e:Ls<}
\lVert \Phi _{L^s}  : \textup{Sparse} (r,r)\rVert \lesssim 2 ^{- \kappa s}, \qquad s\geq 1.  
\end{equation}
\end{lemma}
Combining the 'major arcs' estimate in Lemma \ref{l:major} with the 'minor arcs' estimate in Lemma~\ref{l:minor}, the proof of Theorem \ref{t:main} is complete. 

\bigskip 

The remainder of this section is taken up with the proof of the Lemma. 
The central facts are (1) the Gauss sum bound \eqref{e:Gauss};  (2) the sparse bound for Hilbert space valued  singular integrals Theorem~\ref{t:sparse}, which is applied to Fourier projections of $ f$ and $ g$ onto the major arcs; (3) an argument to pass from a sparse operator applied to the the aforementioned Fourier projections to a sparse bound in terms of just $ f$ and $ g$.  

\smallskip 

\textbf{Step 1}. We define our Hilbert space valued functions, where the Hilbert space will be the finite dimensional space $ \ell ^2 (\mathcal R_s)$.  
Recall that the rationals $ \mathcal R_s$ are defined in \eqref{e:R}, and  the functions $ \chi_s $ are defined in \eqref{e:Ljs}.  
Given $ f \in \ell ^2 $, set 
\begin{equation} \label{e:fs}
f _{s}  =  \{  f _{s, B/Q}  : B/Q\in \mathcal R_s\} :=   \{    \chi _{s-1} \ast (\textup{Mod} _{-B/Q} f)  :   B/Q\in \mathcal R_s\}.  
\end{equation}
Above, $ \textup{Mod} _{\lambda } f (x) = e (\lambda x) f (x)$ is modulation by $ \lambda $.  
The intervals 
\begin{equation} \label{e:disjoint}
 \{ [ B/Q - 10 ^{-s}, B/Q + 10^{-s} ]  : B/Q \in \mathcal R_s\}
\end{equation}
are pairwise disjoint, so that by Bessel's Theorem, we have 
\begin{equation}\label{e:fsf}
\lVert   f_s \rVert _ {\ell ^2 (\ell ^2 (\mathcal R_s) )} 
=
\lVert   \{f_ {s,B/Q}  : B/Q \in \mathcal R_s\} \rVert _ {\ell ^2 (\ell ^2  (\mathcal R_s) )}
\leq \lVert f\rVert _{2}.  
\end{equation}

\smallskip 

\textbf{Step 2.} 
The inner product we are interested in can be viewed as one acting on $ \ell ^2  (\mathcal R_s)$ functions.  
Observe that the Fourier multiplier associated to $ L ^{s}$ enjoys the equalities below. 
Beginning from \eqref{e:Ls} and \eqref{e:Ljs},  
\begin{align}
\langle \Phi _{L ^{s}} f ,g \rangle 
& = \sum_{B/Q\in \mathcal R_s} \sum_{j \geq s/ \varepsilon } 
S (B,Q) \cdot  \langle   H_j (\beta -B/Q)\chi _{s} ( \beta - B/Q  ) \mathcal F f (\beta ) , \mathcal F g (\beta )   \rangle 
\\
&=  \sum_{B/Q\in \mathcal R_s} \sum_{j \geq s/ \varepsilon } 
S (B,Q) \cdot  \langle   H_j (\beta)\chi _{s} ( \beta )   f   (\beta + B/Q) , \mathcal F g  (\beta+B/Q )   \rangle 
\\
&= \sum_{B/Q\in \mathcal R_s} \sum_{j \geq s/ \varepsilon } 
S (B,Q) \cdot  \langle   H_j (\beta)\chi _{s} ( \beta ) \mathcal F f _{s,B/Q} (\beta ) , \mathcal F g _{s, B/Q} (\beta )   \rangle  
\\
\noalign{\noindent  Crucially, above we have removed some modulation factors to get a fixed multiplier acting on a Hilbert space valued function. Continuing the equalities, we have } 
\\ \label{e:TG}
&=   \sum_{B/Q\in \mathcal R_s} S (B,Q)  \langle \Phi _{H^s}  f_{s, B/Q}, g_{s,B/Q}  \rangle,  \qquad \textup{where } \quad   H^s = \sum_{j\geq s/ \varepsilon } H _{j}  . 
\end{align}

We address the Gauss sums $ S (B,Q)$ above. Recalling \eqref{e:Gauss}, and denoting $ f'_s = \{ \lambda _{B/Q} f _{s,B/Q}\}$, for appropriate choice of $ \lvert  \lambda  _{B/Q}\rvert=1 $, we have 
\begin{equation}\label{e:TG1}
\lvert  \langle \Phi _{L^{s}} f_s  , g_s  \rangle \rvert 
\lesssim 2 ^{- \varepsilon  s}  \langle  \Phi_{H ^s} f'_s , g_s \rangle.  
\end{equation}
Above we have gained a geometric decay in $ s$. 

\smallskip

On the right of \eqref{e:TG1}, we have an operator acting on Hilbert space valued functions.  
Noting that $ \lVert f '_{s}\rVert _{\ell ^2 (\mathcal R_s)} =  \lVert f _{s}\rVert _{\ell ^2 (\mathcal R_s)} $ pointwise, 
we are free to replace $ f'_s$ in \eqref{e:TG1} by simply $ f_s$, as defined in \eqref{e:fs}.  
The remaining estimate to prove is that there is a choice of $ 1 < r < 2$, and sparse operator $ \Lambda _{r,r}$ so that 
\begin{equation}\label{e:TG2}
\lvert   \langle  \Phi _{H ^{s}} f_s , g_s \rangle\rvert  \lesssim 2 ^{ \frac \varepsilon 4 s} 
 \Lambda _{r,r} (f,g). 
\end{equation}
Note in particular that we will allow small geometric growth in this estimate, which will be absorbed into the geometric decay in \eqref{e:TG1}.

\smallskip 

\textbf{Step 3.}
The principal step is the application of sparse bound in Theorem~\ref{t:sparse}. 
 From the definitions in \eqref{e:Hj} and \eqref{e:TG}, we have 
\begin{equation*}
H^s (\beta )= \sum_{j\geq s/ \varepsilon } H _{j} (\beta )= \sum_{j \geq s/ \varepsilon } \int e (-\beta t ^{3}) \psi_j (t)
\; dt 
\end{equation*}
By choice of $ \psi $ in \eqref{e:psi}, it follows that the integrand on the right equals $  e (-\beta t ^{3}) \frac {dt}t$ for $ t > 2 ^{ s/ \varepsilon +1}$.  And, in particular, 
\begin{equation*}
H^s (\beta ) = \tfrac 13 \sum_{j \geq s/ \varepsilon } \int e (-\beta s) \frac {\psi_j (s ^{1/3})} {s ^{2/3}} \; ds 
\end{equation*}
But $ \psi  $ is odd,  hence so is $ \frac {\psi_j (s ^{1/3})} {s ^{2/3}}$. 
It follows that $ \check H ^{s}$ is a Calder\'on-Zygmund kernel, that is, it meets the conditions in \eqref{e:CZK}. 
Thus, the operator we are considering is convolution with respect to $ \check H ^{s}$, namely $ \Phi _{H ^{s}} = T _{\check H ^{s}}$. 

Therefore, from Theorem~\ref{t:sparse}, we have  the following inequality for the expression in \eqref{e:TG}: 
\begin{equation}\label{e:TG<}
\lvert   \langle T_{\check H ^{s}}  f_{s}, g_s  \rangle \rvert 
 \lesssim 
\Lambda _{1,1} (  f_s,  g_s).  
\end{equation}
There is one additional fact: All the intervals used in the definition of the sparse form in \eqref{e:TG<} above 
have length at least $ 2 ^{3(s/ \varepsilon-2) }$.  This is a simple consequence of $ \check H ^{s} (x) \mathbf 1_{ [-2 ^{3(s/ \varepsilon-2) }, 2 ^{3(s/ \varepsilon-2) }]} \equiv 0$. 

\smallskip 

\textbf{Step 4.}
We should emphasize that \eqref{e:TG<} has a small abuse of notation: The sparse form is computed on the vector-valued functions $ f_s$ and $ g_s$. That is the implied averages have to be made relative to the $ \ell ^2 (\mathcal R _s)$-norm.  
The last step is to remove the norm.  Namely, we show that there is a choice of $ 1< r < 2$, and sparse form $ \Lambda _{r,r}$ so that 
\begin{equation}\label{e:LL}
\Lambda _{1,1} (   f_s , g_s) \lesssim 2 ^{ \frac \varepsilon 4 s} \Lambda _{r,r} (f,g).  
\end{equation}
Combining this estimate with \eqref{e:TG<}, proves \eqref{e:TG2}, completing the proof.  

\smallskip 

The proof of \eqref{e:LL} is reasonably routine. It will be crucial that we have the estimate $ ^{\sharp} \mathcal R _{s} \lesssim 2 ^{2s}$.  Let $ \mathcal S$ be the sparse collection of intervals associated with the sparse form $ \Lambda _{1,1} (f_s,g_s)$. As noted, we are free to assume that for all $ S\in \mathcal S$, we have $ \lvert  S\rvert \geq 10 ^{s/4 \varepsilon } $.  Recall the the definition of $ f_s$ in \eqref{e:fs}. Write $ f _{s} = f _{s} ^{S,0} + f _{s} ^{S,1}$, where 
\begin{equation*}
f _{s} ^{S,0}    :=      \{ \chi _{s-1} \ast (\textup{Mod} _{-B/Q} (f \mathbf 1_{2S}))  :   B/Q\in \mathcal R_s\}.  
\end{equation*}
Above, we have localized the support of $ f$ to the interval $ 2S$. 
The same decomposition is used on the function $ g$ and $ g_s$. By subadditivity, we have 
\begin{align}\label{e:L0}
\Lambda _{1,1} (   f_s ,  g_s) 
&\leq 
\Lambda _{1,1} (   f_s ^{S,0} ,  g_s ^{S,0})  
\\ \label{e:L1}
& \qquad + \Lambda _{1,1} (   f_s ^{S,1} ,  g_s ^{S,0})  + \Lambda _{1,1} (   f_s ^{S,0} ,  g_s ^{S,1})  
\\ \label{e:L2}
& \qquad + \Lambda _{1,1} (   f_s ^{S,1} ,  g_s ^{S,1}).   
\end{align}

The crux of the matter is this estimate: For each interval $ S\in \mathcal S$, we have 
\begin{equation}\label{e:crux}
\langle  f _{s} ^{S,0} \rangle _{S} \lesssim 2 ^{s \frac {2-r} {r} } \langle f \rangle _{2S,r}, \qquad 1< r < 2. 
\end{equation}
And, the fraction $ \frac {2-r} {r} $ in the exponent can be made arbitrarily small, by taking $ 0< 2-r $ very small.  
Indeed, using the disjointness of the intervals in \eqref{e:disjoint}, and Plancherel, we have 
\begin{equation} \label{e:S0}
\langle  f _{s} ^{S,0} \rangle _{S,2} \lesssim   \langle f \rangle _{2S,2}. 
\end{equation}
Second, it is trivial that 
\begin{equation*}
\langle  \chi _{s-1} \ast (\textup{Mod} _{-B/Q} f \mathbf 1_{2S})  \rangle_S \lesssim \langle f \rangle _{2S}
\end{equation*}
and by simply summing over the bounded number of choices of  $ B/Q\in \mathcal R_s$, we have 
\begin{equation}  \label{e:S1}
\langle  f _{s} ^{S,0} \rangle _{S} \lesssim   2 ^{2s} \langle f \rangle _{2S}. 
\end{equation}
Interpolating between this and \eqref{e:S0} proves \eqref{e:crux}.  With that inequality in hand, we have, for 
$ 0< 2-r$ sufficiently small, 
\begin{align*}
\sum_{S\in \mathcal S} \lvert  S\rvert \langle  f _{s} ^{S,0} \rangle _{S} \langle  g_{s} ^{S,0} \rangle _{S}
& \lesssim 2 ^{ s \frac \varepsilon 4 }
\sum_{S\in \mathcal S} \lvert  S\rvert \langle  f \rangle _{2S,r} \langle  g \rangle _{2S,r}
\end{align*}
If the family $ \mathcal S$ is $ \frac 12 $-sparse, then the family $ \{2S  :  S\in \mathcal S\}$ is $ \frac 14 $-sparse, so we have our desired bound for the term on the right in \eqref{e:L0}.

\medskip 

There are three more terms, in \eqref{e:L1} and \eqref{e:L2}, which  are all much smaller. 
Recall the notation $  \{f\}$ of \eqref{e:new}. 
Since $ \chi $, as chosen in \eqref{e:Ljs}, is smooth, and the length of $ S \in \mathcal S$ is much larger than $ 10 ^{s}$, we have 
\begin{equation*}
\langle  \chi _{s-1} \ast (\textup{Mod} _{-B/Q} f \mathbf 1_{\mathbb R \setminus 2S})  \rangle_S \lesssim 
2 ^{-100s} \{  f \} _{S}, \qquad B/Q\in \mathcal R_s 
\end{equation*}
Summing this estimate over all $ 2 ^{2s}$ choices $ B/Q\in \mathcal R_s$, we see that each of the three terms in \eqref{e:L1} and \eqref{e:L2} are at most 
\begin{equation*}
2 ^{-s}\sum_{S\in \mathcal S} \lvert  S\rvert  \{  f \} _{S} \{  g \} _{S}.  
\end{equation*}
It remains to bound this last bilinear form, which is the task taken up in Lemma~\ref{l:new}. This completes the argument for \eqref{e:LL}.

\bibliographystyle{amsplain}	

\begin{bibdiv}
\begin{biblist}

\bib{2016arXiv160506401B}{article}{
    author={Benea, Cristina},
   author={Bernicot, Fr\'ed\'eric},
   author={Luque, Teresa},
   title={Sparse bilinear forms for Bochner Riesz multipliers and
   applications},
   journal={Trans. London Math. Soc.},
   volume={4},
   date={2017},
   number={1},
   pages={110--128},
   issn={2052-4986},
   review={\MR{3653057}},
   doi={10.1112/tlm3.12005},
}

\bib{MR3531367}{article}{
      author={Bernicot, Fr{\'e}d{\'e}ric},
      author={Frey, Dorothee},
      author={Petermichl, Stefanie},
       title={Sharp weighted norm estimates beyond {C}alder\'on-{Z}ygmund
  theory},
        date={2016},
        ISSN={2157-5045},
     journal={Anal. PDE},
      volume={9},
      number={5},
       pages={1079\ndash 1113},
  url={http://dx.doi.org.prx.library.gatech.edu/10.2140/apde.2016.9.1079},
      review={\MR{3531367}},
}

\bib{MR937581}{article}{
      author={Bourgain, J.},
       title={On the maximal ergodic theorem for certain subsets of the
  integers},
        date={1988},
        ISSN={0021-2172},
     journal={Israel J. Math.},
      volume={61},
      number={1},
       pages={39\ndash 72},
         url={http://dx.doi.org.prx.library.gatech.edu/10.1007/BF02776301},
      review={\MR{937581 (89f:28037a)}},
}

\bib{MR937582}{article}{
      author={Bourgain, J.},
       title={On the pointwise ergodic theorem on {$L\sp p$} for arithmetic
  sets},
        date={1988},
        ISSN={0021-2172},
     journal={Israel J. Math.},
      volume={61},
      number={1},
       pages={73\ndash 84},
         url={http://dx.doi.org.prx.library.gatech.edu/10.1007/BF02776302},
      review={\MR{937582 (89f:28037b)}},
}

\bib{MR1019960}{article}{
      author={Bourgain, Jean},
       title={Pointwise ergodic theorems for arithmetic sets},
        date={1989},
        ISSN={0073-8301},
     journal={Inst. Hautes \'Etudes Sci. Publ. Math.},
      number={69},
       pages={5\ndash 45},
         url={http://www.numdam.org/item?id=PMIHES_1989__69__5_0},
        note={With an appendix by the author, Harry Furstenberg, Yitzhak
  Katznelson and Donald S. Ornstein},
      review={\MR{1019960 (90k:28030)}},
}

\bib{MR3521084}{article}{
      author={Conde-Alonso, Jos{\'e}~M.},
      author={Rey, Guillermo},
       title={A pointwise estimate for positive dyadic shifts and some
  applications},
        date={2016},
        ISSN={0025-5831},
     journal={Math. Ann.},
      volume={365},
      number={3-4},
       pages={1111\ndash 1135},
  url={http://dx.doi.org.prx.library.gatech.edu/10.1007/s00208-015-1320-y},
      review={\MR{3521084}},
}

\bib{2016arXiv160305317C}{article}{
      author={{Culiuc}, A.},
      author={{Di Plinio}, F.},
      author={{Ou}, Y.},
       title={{Domination of multilinear singular integrals by positive sparse
  forms}},
        date={2016-03},
     journal={ArXiv e-prints},
      eprint={1603.05317},
}

\bib{Hua}{book}{
      author={Hua, Loo~Keng},
       title={Introduction to number theory},
   publisher={Springer-Verlag, Berlin-New York},
        date={1982},
        ISBN={3-540-10818-1},
        note={Translated from the Chinese by Peter Shiu},
      review={\MR{665428 (83f:10001)}},
}

\bib{2015arXiv151005789H}{article}{
      author={{Hyt{\"o}nen}, T.~P.},
      author={{Roncal}, L.},
      author={{Tapiola}, O.},
       title={{Quantitative weighted estimates for rough homogeneous singular
  integrals}},
     JOURNAL = {Israel J. Math.},
  FJOURNAL = {Israel Journal of Mathematics},
    VOLUME = {218},
      YEAR = {2017},
    NUMBER = {1},
     PAGES = {133--164},
      ISSN = {0021-2172},
   MRCLASS = {42B25 (42B30)},
  MRNUMBER = {3625128},
MRREVIEWER = {Michael T. Lacey},
       DOI = {10.1007/s11856-017-1462-6},
       URL = {https://doi-org.prx.library.gatech.edu/10.1007/s11856-017-1462-6},
}

\bib{MR3065022}{article}{
      author={Hyt{\"o}nen, Tuomas~P.},
      author={Lacey, Michael~T.},
      author={P{\'e}rez, Carlos},
       title={Sharp weighted bounds for the {$q$}-variation of singular
  integrals},
        date={2013},
        ISSN={0024-6093},
     journal={Bull. Lond. Math. Soc.},
      volume={45},
      number={3},
       pages={529\ndash 540},
         url={http://dx.doi.org/10.1112/blms/bds114},
      review={\MR{3065022}},
}

\bib{ISMW}{article}{
      author={Ionescu, Alexandru~D.},
      author={Stein, Elias~M.},
      author={Magyar, Akos},
      author={Wainger, Stephen},
       title={Discrete {R}adon transforms and applications to ergodic theory},
        date={2007},
        ISSN={0001-5962},
     journal={Acta Math.},
      volume={198},
      number={2},
       pages={231\ndash 298},
  url={http://dx.doi.org.prx.library.gatech.edu/10.1007/s11511-007-0016-x},
      review={\MR{2318564 (2008i:43007)}},
}

\bib{IW}{article}{
      author={Ionescu, Alexandru~D.},
      author={Wainger, Stephen},
       title={{$L\sp p$} boundedness of discrete singular {R}adon transforms},
        date={2006},
        ISSN={0894-0347},
     journal={J. Amer. Math. Soc.},
      volume={19},
      number={2},
       pages={357\ndash 383 (electronic)},
  url={http://dx.doi.org.prx.library.gatech.edu/10.1090/S0894-0347-05-00508-4},
      review={\MR{2188130 (2006g:42019)}},
}

\bib{Iwaniec_Kowalski}{book}{
      author={Iwaniec, Henryk},
      author={Kowalski, Emmanuel},
       title={Analytic number theory},
      series={American Mathematical Society Colloquium Publications},
   publisher={American Mathematical Society, Providence, RI},
        date={2004},
      volume={53},
        ISBN={0-8218-3633-1},
      review={\MR{2061214 (2005h:11005)}},
}

\bib{2016arXiv161103808K}{article}{
      author={{Karagulyan}, Grigor},
       title={{An abstract theory of singular operators}},
        date={2016-11},
     journal={ArXiv e-prints},
      eprint={1611.03808},
}

\bib{2015arXiv151206918K}{article}{
      author={{Krause}, B.},
      author={{Lacey}, M.},
       title={{A Discrete Quadratic Carleson Theorem on $ \ell ^2 $ with a
  Restricted Supremum}},
     journal={Int. Math. Res. Not. IMRN},
   date={2017},
   number={10},
   pages={3180--3208},
   issn={1073-7928},
   review={\MR{3658135}},
   doi={10.1093/imrn/rnw116},
}

\bib{2016arXiv160908701K}{article}{
      author={{Krause}, Ben},
      author={Lacey, Michael~T.},
       title={{Sparse Bounds for Random Discrete Carleson Theorems}},
        date={2016-09},
     journal={ArXiv e-prints},
      eprint={1609.08701},
}

\bib{2016arXiv161001531L}{article}{
      author={{Lacey}, M.~T.},
      author={{Mena}, D.},
       title={{The Sparse T1 Theorem}},
     journal={Houston J. Math.},
   volume={43},
   date={2017},
   number={1},
   pages={111--127},
   issn={0362-1588},
   review={\MR{3647935}},
}

\bib{2015arXiv150105818L}{article}{
      author={{Lacey}, Michael~T.},
       title={{An elementary proof of the $A\_2$ Bound}},
     journal={Israel J. Math.},
   volume={217},
   date={2017},
   number={1},
   pages={181--195},
   issn={0021-2172},
   review={\MR{3625108}},
   doi={10.1007/s11856-017-1442-x},
}

\bib{2016arXiv160906364L}{article}{
      author={Lacey, Michael~T.},
      author={Spencer, Scott},
       title={{Sparse Bounds for Oscillatory and Random Singular Integrals}},
    journal={New York J. Math.},
   volume={23},
   date={2017},
   pages={119--131},
   issn={1076-9803},
   review={\MR{3611077}},
}

\bib{MR3085756}{article}{
      author={Lerner, Andrei~K.},
       title={A simple proof of the {$A_2$} conjecture},
        date={2013},
        ISSN={1073-7928},
     journal={Int. Math. Res. Not. IMRN},
      number={14},
       pages={3159\ndash 3170},
      review={\MR{3085756}},
}

\bib{2015arXiv151207524M}{article}{
      author={{Mirek}, M.},
       title={{Square function estimates for discrete Radon transforms}},
    journal={Anal. PDE},
   volume={11},
   date={2018},
   number={3},
   pages={583--608},
   issn={2157-5045},
   review={\MR{3738256}},
   doi={10.2140/apde.2018.11.583},
}

\bib{2015arXiv151207518M}{article}{
      author={{Mirek}, M.},
      author={{Stein}, E.~M.},
      author={{Trojan}, B.},
       title={{L\^{}p(Z\^{}d)-estimates for discrete operators of Radon type:
  Maximal functions and vector-valued estimates}},
        date={2015-12},
     journal={ArXiv e-prints},
      eprint={1512.07518},
}

\bib{2015arXiv151207523M}{article}{
  author={Mirek, Mariusz},
   author={Stein, Elias M.},
   author={Trojan, Bartosz},
   title={$\ell^p(\Bbb Z^d) $-estimates for discrete operators of Radon
   type: variational estimates},
   journal={Invent. Math.},
   volume={209},
   date={2017},
   number={3},
   pages={665--748},
   issn={0020-9910},
   review={\MR{3681393}},
   doi={10.1007/s00222-017-0718-4},
}

\bib{MR2661174}{article}{
      author={Pierce, Lillian~B.},
       title={A note on twisted discrete singular {R}adon transforms},
        date={2010},
        ISSN={1073-2780},
     journal={Math. Res. Lett.},
      volume={17},
      number={4},
       pages={701\ndash 720},
  url={http://dx.doi.org.prx.library.gatech.edu/10.4310/MRL.2010.v17.n4.a10},
      review={\MR{2661174 (2011m:42019)}},
}

\bib{MR1056560}{article}{
      author={Stein, Elias~M.},
      author={Wainger, Stephen},
       title={Discrete analogues of singular {R}adon transforms},
        date={1990},
        ISSN={0273-0979},
     journal={Bull. Amer. Math. Soc. (N.S.)},
      volume={23},
      number={2},
       pages={537\ndash 544},
  url={http://dx.doi.org.prx.library.gatech.edu/10.1090/S0273-0979-1990-15973-7},
      review={\MR{1056560 (92e:42010)}},
}

\end{biblist}
\end{bibdiv}

 \end{document}